\documentclass{article}

\usepackage{amsmath,amssymb,mathrsfs}
\usepackage{amsthm}
\usepackage{dsfont}
\usepackage{graphicx}
\usepackage{empheq}
\usepackage{faktor}
\usepackage{xcolor}
\usepackage{enumitem}
\usepackage{hyperref}

\usepackage[english]{babel}
\usepackage[T1]{fontenc}
\usepackage[utf8]{inputenc}

\newcommand{\dis}{\displaystyle}
\newcommand{\lb}{\lambda}
\newcommand{\p}{\partial}
\newcommand{\beq}{\begin{equation}}
\newcommand{\eeq}{\end{equation}}
\renewcommand{\epsilon}{\varepsilon}

\renewcommand{\leq}{\leqslant}
\renewcommand{\geq}{\geqslant}
\renewcommand{\d}{\mathrm{d} }
\newcommand{\R}{\mathbb{R}}
\newcommand{\C}{\mathbb{C}}
\newcommand{\Z}{\mathbb{Z}}
\newcommand{\N}{\mathbb{N}}
\newcommand{\E}{\mathbb{E}}

\newcommand{\ddt}{\frac{\p}{\p t}}
\newcommand{\ddx}{\frac{\p}{\p x}}

\newcommand{\supp}{\mathrm{supp \,}}

\newcommand{\U}{\mathcal U}

\newcommand{\M}{\mathcal{M}}
\newcommand{\A}{\mathcal A}
\newcommand{\e}{\mathrm e}
\newcommand{\intRpe}{\int_0^\infty}
\newcommand{\Mdot}{{\dot{\M}}}
\newcommand{\Cdot}{{\dot{\mathcal{C}}}}
\newcommand{\xk}{x^{-\frac{2ik\pi}{\log 2}}}
\newcommand{\Rpe}{{(0,\infty)}}

\newcommand{\B}{\mathcal{B}}
\newcommand{\Bloc}{{\B_{loc}}}

\newcommand{\fdt}{\frac{\mathrm{d}}{\mathrm{d}t}}
\newcommand{\f}{\frac}
\newcommand{\1}{\mathds 1}
\newcommand{\Cr}{\mathcal{C}}
\newcommand{\X}{\mathbf X}

\renewcommand{\Re}{\mathrm{Re}}
\renewcommand{\Im}{\mathrm{Im}}

\newtheorem{defi}{Definition}
\newtheorem{theorem}{Theorem}
\newtheorem*{theorem*}{Theorem}
\newtheorem{prop}[theorem]{Proposition}
\newtheorem{lem}[theorem]{Lemma}
\newtheorem{coro}[theorem]{Corollary}

\title{Periodic asymptotic dynamics of the measure solutions to an equal mitosis equation}

\date{}

\author{Pierre Gabriel \thanks{Laboratoire de Math\'ematiques de Versailles, UVSQ, CNRS, Universit\'e Paris-Saclay,  45 Avenue des \'Etats-Unis, 78035 Versailles cedex, France. Email: pierre.gabriel@uvsq.fr}
\and Hugo Martin \thanks{Laboratoire Jacques-Louis Lions, CNRS UMR 7598, Sorbonne universit\'e, 4 place Jussieu, 75005 Paris, France. Corresponding author, Email: hugo.martin@sorbonne-universite.fr}}

\begin{document}

\maketitle

\begin{abstract}
We are interested in a non-local partial differential equation modeling equal mitosis.
We prove that the solutions present persistent asymptotic oscillations and that the convergence to this periodic behavior, in suitable spaces of weighted signed measures, occurs exponentially fast.
It can be seen as a result of spectral gap between the countable set of dominant eigenvalues and the rest of the spectrum, which is to our knowledge completely new.
The two main difficulties in the proof are to define the projection onto the subspace of periodic (rescaled) solutions and to estimate the speed of convergence to this projection.
The first one is addressed by using the generalized relative entropy structure of the dual equation, and the second is tackled by applying Harris's ergodic theorem on sub-problems.
\end{abstract}

\


\


\section{Introdution}

We are interested in the following nonlocal transport equation
\begin{equation}\label{eq:mitosis}
\ddt u(t,x) + \ddx \big(x\,u(t,x)\big) + B(x)u(t,x) = 4B(2x)u(t,2x),\qquad x>0,
\end{equation}
which appears as an idealized size-structured model for the bacterial cell division cycle~\cite{BellAnderson,SinkoStreifer71}.
The unknown $u(t,x)$ represents the population density of cells of size $x$ at time $t,$ which evolves according to two phenomena:
the individual exponential growth which results in the transport term $\partial_x(xu(t,x)),$ and the division into two equal-sized daughter cells with rate $B(x)$ that corresponds to the nonlocal infinitesimal term $2B(2x)u(t,2x)2\d x-B(x)u(t,x)\d x.$

\

Equation~\eqref{eq:mitosis} is an interesting and challenging critical case of the general linear growth-fragmentation equation, which reads
\[\dis \ddt u(t,x) + \ddx \big(g(x)u(t,x)\big) + B(x)u(t,x) = \int_x^\infty B(y)u(t,y)k(y,\d x)\,\d y.\]
The long time behavior of this equation is strongly related to the existence of steady size distributions, namely solutions of the form $\U_0(x)\e^{\lambda_0 t}$ with $\U_0$ nonnegative and integrable.
It is actually equivalent to say that $\U_0$ is a Perron eigenfunction associated to the eigenvalue $\lambda_0$.
Such an eigenpair $(\lambda_0,\U_0)$ typically exists when, roughly speaking, the fragmentation rate $B$ dominates the growth speed $g$ at infinity and on the contrary $g$ dominates $B$ around the origin (see~\cite{DoumicEscobedo,DG,Doumic2018} for more details).
In most cases where this existence holds, the solutions behave asymptotically like the steady size distribution $\U_0(x)\e^{\lambda_0 t}.$
This property, known as asynchronous exponential growth~\cite{WebbGrabosch}, has been proved by many authors using various methods since the pioneering work of Diekmann, Heijmans, and Thieme~\cite{Diekmann1984}.
Most of these results focus on one of the two special cases $g(x)=1$ (linear individual growth) or $g(x)=x$ (exponential individual growth).
When $g(x)=1$ it has been proved for the equal mitosis case $k(y,\d x)=2\delta_{x=\frac y2}$ or more general kernels $k(y,\d x)$ by means of spectral analysis of semigroups~\cite{BernardGabriel2,Diekmann1984,GreinerNagel,Heijmans1984,MischlerScher}, general relative entropy method~\cite{Debiec2018,Michel2005} and/or functional inequalities~\cite{Balague2013,Caceres2011,LaurencotPerthame,Monmarche,Pakdaman2014,PerthameRhyzik}, the theory of stochastic semigroups~\cite{BanasiakPichorRudnicki,Bansaye2019,Bouguet,Broda,CanizoGabrielYoldas,CanizoYoldas,RudnickiPichor}, coupling arguments~\cite{BCGMZ,ChafaiMalrieuParoux,Malrieu}, many-to-one formula~\cite{Cloez}, or explicit expression of the solutions~\cite{Zaidi2015}.
For the case $g(x)=x$ asynchronous exponential growth is proved under the assumption that for all $y>0$ the kernel $k(y,\d x)$ has an absolutely continuous part with respect to the Lebesgue measure:
by means of spectral analysis of semigroups~\cite{BernardGabriel2,Heijmans1984,MischlerScher}, general relative entropy method~\cite{Escobedo2005,Michel2005} and/or functional inequalities~\cite{Balague2013,Caceres2010,Caceres2011,GabrielSalvarani}, Foster-Lyapunov criteria~\cite{Bouguet,CanizoGabrielYoldas}, Feynman-Kac~\cite{Bertoin2019,BertoinWatson2019,BertoinWatson2018,Cavalli} or many-to-one formulas~\cite{Marguet2}.

\

The assumption that the fragmentation kernel has a density part when $g(x)=x$ is not a technical but a necessary restriction.
In the equal mitosis case of Equation~\eqref{eq:mitosis} for instance, asynchronous exponential growth does not hold.
It can be easily understood by noticing that if at time $t=0$ the population is composed of only one individual with deterministic size $x>0,$ then for any positive time $t$ the size of the descendants necessarily belongs to the set $\{x\e^t2^{-n}: n\in\N\}$.
This observation was made already by Bell and Anderson in~\cite{BellAnderson} and it has two important consequences.

First the solution $u(t,x)$ cannot relax to a steady size distribution and it prevents Equation~\eqref{eq:mitosis} from having the asynchronous exponential growth property.
The dynamics does not mix enough the trajectories to generate ergodicity, and the asymptotic behavior keeps a strong memory of the initial data.
This situation has been much less studied than the classical ergodic case.
In~\cite{Diekmann1984,Heijmans1985} Diekmann, Heijmans, and Thieme made the link with the existence of a nontrivial boundary spectrum: all the complex numbers $\lambda_k=1+\frac{2ik\pi}{\log2},$ with $k$ lying in $\Z,$ are eigenvalues.
As a consequence the Perron eigenvalue $\lambda_0=1$ is not strictly dominant and it results in persistent oscillations, generated by the boundary eigenfunctions.
The convergence to this striking behavior was first proved in~\cite{GreinerNagel} in the space $L^1([\alpha,\beta])$ with $[\alpha,\beta]\subset(0,\infty).$
More recently it has been obtained in $L^2((0,\infty),x/\U(x)\,\d x)$~\cite{BDG}, and also in $L^1((0,\infty),x^{1+r}\d x)$ for monomial division rates $B(x)=x^r$ ($r>0$) and smooth initial data~\cite{vanBrunt2018}.

Second, it highlights the lack of regularizing effect of the equation.
If the initial distribution is a Dirac mass, then the solution is a Dirac comb for any time.
It contrasts with the cases of density fragmentation kernels for which the singular part of the measure solutions vanishes asymptotically when times goes to infinity~\cite{Debiec2018},
and gives an additional motivation for studying Equation~\eqref{eq:mitosis} in a space of measures.

In the present paper, we prove the convergence to asymptotic oscillations for the measure solutions of Equation~\eqref{eq:mitosis} in weighted total variation distance.
Besides, we get an exponential rate of convergence, which is a major novelty compared to the previous results in~\cite{BDG,GreinerNagel,vanBrunt2018}.

\

Before stating precisely the main results of the paper, let us start with some definitions.
For a Borel weight function $w:(0,\infty)\to(0,\infty)$, we denote by $\M(w)$ the space of (real-valued) Radon measures $\mu=\mu_+-\mu_-$ on $(0,\infty)$ such that
\[\left\|\mu\right\|_{\M(w)}:=\int_{(0,\infty)}\! w\,\d|\mu|<\infty,\]
where $|\mu|=\mu_++\mu_-$ is the total variation of $\mu$.
We use the shorthand $\dot\M$ for the case $w(x)=x$, which provides a natural space to look for measure solutions to Equation~\eqref{eq:mitosis}.
Indeed a formal integration against the measure $x\,\d x$ over $(0,\infty)$ yields the balance law 
\[\int_0^\infty xu(t,x)\d x = \e^t\int_0^\infty xu(0,x)\d x.\]
We denote by $\B(w)$ the space of Borel functions $f : (0,\infty) \to \R$ such that
\[\left\|f\right\|_{\B(w)} := \sup_{x>0}\,\frac{|f(x)|}{w(x)}<\infty.\]
An element $\mu$ of $\M(w)$ defines a linear form on $\B(w)$ through
\[\langle\mu,f\rangle:=\int_{(0,\infty)}\! f\,\d\mu:=\int_{(0,\infty)}\! f\,\d\mu_+-\int_{(0,\infty)}\! f\,\d\mu_-.\]
We also define the subset $\Cr(w)\subset\B(w)$ of continuous functions, and the subset $\Cr_0(w)\subset \Cr(w)$ of the functions such that the ratio $f(x)/w(x)$ vanishes at zero and infinity.
The weighted total variation norm is also given by the dual representation
\[\|\mu\|_{\M(w)}=\sup_{\|f\|_{\B(w)}\leq1}\langle\mu,f\rangle\]
where the supremum can be indifferently taken over $\B(w)$, $\Cr_0(w)$ or $\Cr_c(0,\infty)$.
As for $\dot\M$, we denote by $\dot\B,\Cdot,\Cdot_0$ the spaces $\B(w),\Cr(w),\Cr_0(w)$ corresponding to the choice $w(x)=x$.

\

Now we state the notion of solutions that we will use for Equation~\eqref{eq:mitosis} in the space~$\dot\M$.
First we define the operator $\A$ acting on the space $\Cr^1(0,\infty)$ of continuously differentiable functions via
\[\A f(x):=xf'(x)+B(x)\big(2f(x/2)-f(x)\big)\]
and its domain
\[D(\A):=\big\{f\in \dot\B\cap\Cr^1(0,\infty):\A f\in\dot\B\big\}.\]
The definition we choose for the measure solutions to Equation~\eqref{eq:mitosis} is of the ``mild'' type in the sense that it relies on an integration in time, and of the ``weak'' type in the sense that it involves test functions in space.
\begin{defi}\label{def:measure sol}
A family $(\mu_t)_{t\geq 0}\subset \Mdot$ is called a measure solution to Equation~\eqref{eq:mitosis} if for all $f\in \Cdot$ the mapping $t \mapsto \langle\mu_t,f\rangle$ is continuous, and for all $t\geq 0$ and all $f\in D(\A)$
\begin{equation}\label{eq:mitosismeasure}
\langle\mu_t,f\rangle=\langle\mu_0,f\rangle+\int_0^t\langle\mu_s,\A f\rangle\,\d s.
\end{equation}
\end{defi}
In order to prove the well-posedness of Equation~\eqref{eq:mitosis} in the sense of this definition, we will make the following assumption on the division rate
\begin{equation}\label{hyp:Bsol}
B:(0,\infty)\to[0,\infty)\ \text{ is continuous and bounded around }0.
\end{equation}

\

Let us now look at the (direct and dual) eigenvalue problem associated to Equation~\eqref{eq:mitosis}.
It consists in finding $\lambda\in\C$ together with nonzero $\U$ and $\phi$ such that
\begin{align}
&\big(x\,\U(x)\big)' + \left(B(x)+\lb\right) \U(x) = 4B(2x)\U(2x),\label{Perron}
\vspace{4mm}\\
& -x\phi'(x) + \left(B(x)+\lb\right) \phi(x) = 2B(x)\phi\Big(\f x2\Big).\label{Perronadj}
\end{align}
Requiring that $\U$ and $\phi$ are non-negative corresponds to the Perron eigenvalue problem, and it has been solved under various assumptions on the division rate $B$ in~\cite{DG,Escobedo2005,Hall1990,Michel2006}.
The most general result is the one obtained as a particular case of~\cite[Theorem~1]{DG}, which guarantees that if $B$ satisfies~\eqref{hyp:Bsol} and the additional assumption
\begin{equation}\label{hyp:B}
\left\{\begin{array}{l}
\supp B=[b,+\infty) \text{ for some } b\geq0,
\vspace{2mm}\\
 \exists b_0,\gamma_0,K_0>0,\ \forall x<b_0,\quad B(x)\leq K_0 x^{\gamma_0}
 \vspace{2mm}\\
\exists b_1,\gamma_1,\gamma_2,K_1,K_2>0\ \forall x>b_1\quad K_1x^{\gamma_1} \leq B(x) \leq K_2x^{\gamma_2},
\end{array}\right.
\end{equation}
then there exists a unique nonnegative eigenfunction $\U_0\in L^1(0,\infty)$ solution to~\eqref{Perron} and normalized by $\int_0^\infty x\,\U_0(x)\d x=1.$
It is associated to the eigenvalue $\lambda_0=1$ and to the dual eigenfunction $\phi_0(x)=x$ solution to~\eqref{Perronadj}.
Moreover,
\[\forall r\in \R, \qquad x^r\,\U_0\in L^1\Rpe \cap L^\infty\Rpe.\]
As already noticed in~\cite{Diekmann1984} (see also example 2.15, p.354 in~\cite{Nagel86}), the Perron eigenvalue $\lambda_0$ is not strictly dominant in the present case.
There is an infinite number of eigenvalues with real part equal to $1$.
More precisely for all $k\in\Z$ the triplet $(\lambda_k,\U_k,\phi_k)$ defined from $(\lambda_0,\U_0,\phi_0)$ by 
\[\lb_k = 1+\frac{2ik\pi}{\log 2}, \qquad \U_k(x) = \xk \U_0(x), \qquad \phi_k(x) = x^{1+\frac{2ik\pi}{\log 2}},\]
verifies~\eqref{Perron}-\eqref{Perronadj}.
In such a situation the asynchronous exponential growth property cannot hold, since for any $k\in\Z\setminus\{0\}$ the functions $\Re\big(\U_k(x)\e^{\lambda_kt}\big)$ and $\Im\big(\U_k(x)\e^{\lambda_kt}\big)$
are solutions to Equation~\eqref{eq:mitosis} that oscillate around $\U_0(x)\e^t$.
Since we work in $\dot\M$, it will be useful to define the complex-valued Radon measures
\[\nu_k(\d x):=\U_k(x)\d x.\]
The main results of the paper are summarized in the following theorem.

\begin{theorem}\label{thm:main}
Let $\mu_0\in\dot\M$.
If Assumption~\eqref{hyp:Bsol} is verified, then there exists a unique measure solution $(\mu_t)_{t\geq0}$ to Equation~\eqref{eq:mitosis} in the sense of Definition~\ref{def:measure sol}.
If additionally $B$ satisfies~\eqref{hyp:B}, then there exists a unique $\log 2$-periodic family $(\rho_t)_{t\geq0}\subset \Mdot$ such that for all $f\in \Cdot_0$
\[\langle\e^{-t}\mu_t,f\rangle-\langle\rho_t,f\rangle\xrightarrow[t\to+\infty]{}0.\]
Moreover, for any $t\geq0$, the Radon measure $\rho_t$ is characterized through a Fejér type sum:
for all $f\in \Cr^1_c(0,\infty)$
\[\langle\rho_t,f\rangle=\lim_{N\to\infty}\sum_{k=-N}^{N}\Big(1-\frac{|k|}{N}\Big)\langle\mu_0,\phi_k\rangle\langle\nu_k,f\rangle\e^{\frac{2ik\pi}{\log2}t}.\]
Finally, consider two real numbers $r_1$ and $r_2$ with $r_1<1<r_2$ and define the weight $w(x)=x^{r_1}+x^{r_2}$.
If $\mu_0$ belongs to $\M(w)$ then so does $\rho_0$ and there exist computable constants $C\geq1$ and $a>0$, that depend only on $r_1,r_2$ and $B$, such that for all $t\geq0$
\[\left\|\e^{-t}\mu_t-\rho_t\right\|_{\M(w)}\leq C\e^{-at}\left\|\mu_0-\rho_0\right\|_{\M(w)}.\]
\end{theorem}

\

Let us make some comments about these results:
\begin{enumerate}[label=(\roman*)]
\item It is worth noticing that the well-posedness of Equation~\eqref{eq:mitosis} does not require any upper bound for the division rate.
It contrasts with existing results in Lebesgue spaces where at most polynomial growth is usually assumed.
\item In~\cite{BDG,GreinerNagel,vanBrunt2018} the convergence to the oscillating behavior is proved to occur in norm but without any estimate on the speed. 
Here we extend the convergence to measure solutions and provide for the first time an explicit rate of decay in suitable weighted total variation norms.
\item In~\cite{BDG}, the dynamic equilibrium is characterized as a Fourier type series.
In our result it is replaced by a Fejér sum, namely the Cesàro means of the Fourier series.
\item Even though all the $\nu_k$ have a density with respect to the Lebesgue measure, the limit $\rho_t$ does not in general.
Indeed, as already noticed, if for instance $\mu_0=\delta_x$ then $\supp\mu_t\subset\{x\e^t2^{-n}: n\in\N\}$.
Consequently $\supp\rho_t\subset\{x\e^t2^{-n}: n\in\Z\}$ and $\rho_t$ is thus a Dirac comb.
\item We easily notice in the explicit formula of $\rho_t$ that if $\mu_0$ is such that $\langle\mu_0,\phi_k\rangle=0$ for all $k\neq0$,
then there is no oscillations and the solution behaves asymptotically like $\U_0(x)\e^t$, similarly to the asynchronous exponential growth case.
Such initial distributions actually do exist, as for instance the one proposed in~\cite{vanBrunt2018} which reads in our setting
\[\mu_0(\d x)=\frac{1}{x^2}\1_{[1,2]}(x)\,\d x\]
where $\1_{[1,2]}$ denotes the indicator function of the interval $[1,2].$
\end{enumerate}

\

The proof of the results given in Theorem~\ref{thm:main} is split into the two next sections:
Section \ref{sec:wellposedness} is about the well-posedness of Equation~\eqref{eq:mitosis} in the framework of measure solutions, and
Section \ref{sec:asymp} is devoted to the analysis of the long time asymptotic behavior.
In a concluding section, we draw some future directions that can extend the present work.

\section{Well-posedness in the measure setting} \label{sec:wellposedness}

Measure solutions to structured populations dynamics PDEs have attracted increasing attention in the last few years, and there exist several general well-posedness results~\cite{CanizoCarrilloCuadrado,Carrillo2012,EversHilleMuntean15,EversHilleMuntean16,Gwiazda2010}.
However they do not apply here due to the unboundedness of the function $B$.
We overcome this difficulty by adopting a duality approach in the spirit of~\cite{Bansaye2017,Bansaye2019,Dumont2017,Gabriel2018},
which also proves useful for investigating the long time behavior.

\

Additional motivation to consider measure solutions lies in the study of stochastic processes.
Equation~\eqref{eq:mitosis} is the Kolmogorov (forward) equation of the underlying piecewise deterministic branching process.
Let us give here a brief informal description of this, and refer to~\cite{Campillo,Cloez,DerfelVanBruntWake,DHKR,Marguet1} for more details.
Take the measure-valued branching process $(Z_t)_{t\geq0}$ defined as the empirical measure
\[Z_t=\sum_{i\in V_t}\delta_{X^i_t}\]
where $V_t$ is the set of individuals alive at time $t$ and $\{X^i_t : i\in V_t\}$ the set of their sizes.
For each individual $i\in V_t$ the size $X^i_t$ grows exponentially fast following the deterministic flow $\frac{\d}{\d t}X^i_t=X^i_t$ until a division time $T_i$ which occurs stochastically with rate $B(X^i_t).$
Then the individual $i$ dies and gives birth to two daughter cells $i_1$ and $i_2$ with size $X^{i_1}_{T_i}=X^{i_2}_{T_i}=\frac12 X^i_{T_i}.$
Taking the expectancy of the random measures $Z_t,$ we get a family of measures $u(t,\cdot)$ defined for any Borel set $A\subset(0,\infty)$ by
\[u(t,A):=\E[Z_t(A)]=\E\big[\#\{i\in V_t:X^i_t\in A\}\big],\]
which is a weak solution to the Kolmogorov Equation~\eqref{eq:mitosis}, see {\it e.g.} \cite[Theorem~1]{DHKR} or~\cite[Corollary~5.1]{Marguet1}.

Another Kolmogorov equation is classically associated to $(Z_t)_{t\geq0},$ which is the dual equation of~\eqref{eq:mitosis}
\begin{equation}\label{eq:dual}
\ddt \varphi(t,x) = x \ddx \varphi(t,x) + B(x)\big[2\varphi(t,x/2)-\varphi(t,x)\big],\qquad x>0.
\end{equation}
This second equation is sometimes written in its backward version where $\ddt \varphi(t,x)$ is replaced by $-\ddt \varphi(t,x),$ and is then usually called the Kolmogorov backward equation.
Nevertheless since the division rate $B(x)$ does not depend on time, we prefer here writing this backward equation in a forward form so that for any observation function $f$
\[\varphi(t,x):=\E[\langle Z_t,f\rangle\,|\,Z_0=\delta_x]=\E\bigg[\sum_{i\in V_t}f(X^i_t)\,\big|\,Z_0=\delta_x\bigg]\]
is the solution to~\eqref{eq:dual} with initial data $\varphi(0,x)=f(x)$, see {\it e.g.} \cite[Proposition~4]{Campillo} or~\cite[Corollary~2.4]{Cloez}.

Equation~\eqref{eq:mitosis} is then naturally defined on a space of measure, while Equation~\eqref{eq:dual} is defined on a space of functions.

\

To prove the well-posedness of Equation~\eqref{eq:mitosis} in $\dot\M$, we proceed by duality and first build a semigroup on~$\dot\B$ that yields the unique solutions to Equation~\eqref{eq:dual}.

\subsection{The dual equation}\label{subsec:adj}

We start by proving the well-posedness of Equation~\eqref{eq:dual} in a space larger than~$\dot\B$, which turns out to be more suitable.
For a subset $\Omega \subset \R^d$, we denote by $\Bloc(\Omega)$ the space of functions $f : \Omega \to \R$ that are bounded on $\Omega\cap B(0,r)$ for any $r>0$, and by $\B(\Omega)$ the (Banach) subspace of bounded functions endowed with the supremum norm $\|f\|_\infty=\sup_{x\in\Omega}|f(x)|$.
Using these spaces allows us to prove the well-posedness without needing any upper bound at infinity on the division rate~$B$.

In the following proposition, we prove that for any $f\in\Bloc(0,\infty)$ there exists a unique solution $\varphi\in\Bloc([0,\infty)\times(0,\infty))$ to Equation~\eqref{eq:dual} in a mild sense (Duhamel formula) with initial condition $\varphi(0,\cdot)=f$.
Moreover, we show that if $f\in \Cr^1(0,\infty)$ then $\varphi$ is also continuously differentiable and verifies Equation~\eqref{eq:dual} in the classical sense.

\begin{prop}\label{lm:fixedpoint}
Assume that $B$ satisfies~\eqref{hyp:Bsol}.
Then for any $f\in\B_{loc}(0,\infty)$ there exists a unique $\varphi\in\B_{loc}([0,\infty)\times(0,\infty))$ such that for all $t\geq0$ and $x>0$
\[\varphi(t,x) = f(x\e^t)\e^{-\int_0^t B(x\e^s)\d s} + 2\int_0^t B(x\e^\tau)\e^{-\int_0^\tau B(x\e^s)\d s}\varphi\Big(t-\tau,\frac{x\e^\tau}{2}\Big)\d \tau.\]
Moreover if $f$ is nonnegative/continuous/continuously differentiable, then so is~$\varphi.$
In the latter case $\varphi$ verifies for all $t,x>0$
\[\ddt \varphi(t,x)= \A\varphi(t,\cdot)(x)=x \ddx \varphi(t,x) + B(x)\big(2\varphi(t,x/2)-\varphi(t,x)\big).\]
\end{prop}

\medskip

\begin{proof}
Let $f\in\B_{loc}(0,\infty)$ and define on $\B_{loc}([0,\infty)\times(0,\infty))$ the mapping $\Gamma$ by
\[\Gamma g(t,x)=f(x\e^t)\e^{-\int_0^t B(x\e^s)\,\d s} + 2\int_0^t B(x\e^\tau)\e^{-\int_0^\tau B(x\e^s)\,\d s} g\Big(t-\tau,\frac{x\e^\tau}{2}\Big)\,\d\tau.\]
For $T,K>0$ define the set $\Omega_{T,K}=\{(t,x)\in[0,T]\times(0,\infty),\ x\e^t<K\}.$
Clearly $\Gamma$ induces a mapping $\B(\Omega_{T,K})\to\B(\Omega_{T,K}),$ still denoted by $\Gamma.$
To build a fixed point of $\Gamma$ in $\B_{loc}([0,\infty)\times(0,\infty))$ we prove that it admits a unique fixed point in any $\B(\Omega_{T,K})$, denoted $\varphi_{T,K}$, that we will build piecewisely on subsets of $\Omega_{T,K}$.

\medskip

Let $K>0$ and $t_0<1/(2\sup_{(0,K)}B).$
For any $g_1,g_2\in\B(\Omega_{t_0,K})$ we have
\[\left\|\Gamma g_1-\Gamma g_2\right\|_\infty\leq2t_0 \sup_{(0,K)}\!B\,\left\|g_1-g_2\right\|_\infty\]
and $\Gamma$ is a contraction.
The Banach fixed point theorem then guarantees the existence of a unique fixed point $\varphi^0=\varphi_{t_0,K}$ of $\Gamma$ in $\B(\Omega_{t_0,K}).$
To extend this construction up to $T$, we cover $\Omega_{T,K}$ with strips of width $t_0$:
\[
\omega_k := \left\{(t,x)\in [kt_0,(k+1)t_0]\times (0,\infty),\ x\e^t < K\right\}
\]
where $k\in\N$.
We start by setting ${\varphi_{T,K}}_{|\omega_0} := \varphi^0$.
Then, applying the fixed point technique in $\B(\Omega_{t_0,K\e^{-t_0}})$ with $\varphi^0(t_0,\cdot)$ instead of $f$ yields a function~$\varphi^1$.
Defining ${\varphi_{T,K}}_{|\omega_1} := \varphi^1(\cdot + t_0,\cdot)$, we get a unique fixed point of $\Gamma$ on $\omega_0 \cup \omega_1$ and, repeating the argument, a unique fixed point of $\Gamma$ on $\Omega_{T,K}$.

\medskip

For $T'>T>0$ and $K'>K>0$ we have ${\varphi_{T',K'}}_{|\Omega_{T,K}}=\varphi_{T,K}$ by uniqueness of the fixed point in $\B(\Omega_{T,K}),$
and we can define $\varphi$ by setting $\varphi_{|\Omega_{T,K}}=\varphi_{T,K}$ for any $T,K>0.$
Clearly the function $\varphi$ thus defined is the unique fixed point of $\Gamma$ in $\B_{loc}([0,\infty)\times(0,\infty)).$

\medskip

Since $\Gamma$ preserves the closed cone of nonnegative functions if $f$ is nonnegative, the fixed point $\varphi_{t_0,K}$ is necessarily nonnegative when $f$ is so.
Then by iteration $\varphi_{T,K}\geq0$ for any $T,K>0,$ and ultimately $\varphi\geq0$.
Similarly, the closed subspace of continuous functions being invariant under $\Gamma$ when $f$ is continuous, the fixed point $\varphi$ inherits the continuity of $f$.

\medskip

Consider now that $f$ is continuously differentiable on $(0,\infty).$
Unlike the sets of nonnegative or continuous functions, the subspace $\Cr^1(\Omega_{t_0,K})$ is not closed in $\B(\Omega_{t_0,K})$ for the norm $\|\cdot\|_\infty.$
For proving the continuous differentiability of $\varphi$ we repeat the fixed point argument in the Banach spaces
\[\{g\in \Cr^1(\Omega_{T,K}),\ g(0,\cdot)=f\}\]
endowed with the norm
\[\left\|g\right\|_{\Cr^1}=\left\|g\right\|_\infty+\left\|\partial_tg\right\|_\infty+\left\|x\partial_xg\right\|_\infty.\]
Differentiating $\Gamma g$ with respect to $t$ we get
\begin{align}
\p_t(\Gamma g)(t,x)&=\left[x\e^tf'(x\e^t) - B(x\e^t)f(x\e^t)\right]\e^{-\int_0^t B(x\e^s)\d s} \nonumber\\
&\hspace{16mm}+ 2B(x\e^t)\e^{-\int_0^t B(x\e^s)\d s}g\Big(0,\frac{x\e^t}{2}\Big) \nonumber\\
&\hspace{24mm} + 2\int_0^t B(x\e^\tau)\e^{-\int_0^\tau B(x\e^s)\d s} \p_t g\Big(t-\tau,\frac{x\e^\tau}{2}\Big)\d \tau \nonumber\\
	&\hspace{-10mm}=\A f(x\e^t)\e^{-\int_0^t B(x\e^s)\d s} + 2\int_0^t B(x\e^\tau)\e^{-\int_0^\tau B(x\e^s)\d s} \p_t g\Big(t-\tau,\frac{x\e^\tau}{2}\Big)\d \tau \label{eq:dtGammag}
\end{align}
and differentiating the alternative formulation
\[\Gamma g(t,x)=f(x\e^t)\e^{-\int_x^{x\e^t} B(z)\frac{\d z}{z}} + 2\int_x^{x\e^t} B(y)\e^{-\int_x^y B(z)\f{\d z}{z}}g\Big(t-\log\big(\frac{y}{x}\big),\frac{y}{2}\Big) \f{\d y}{y}\]
with respect to $x$ we obtain
\begin{align*}
x\partial_x(\Gamma g)(t,x)&=\Big[\A f(x\e^t)+B(x)f(x\e^t)\Big]\e^{-\int_x^{x\e^t} B(z)\frac{\d z}{z}}-2B(x)g\Big(t,\frac x2\Big)\\
&\qquad+2B(x)\int_x^{x\e^t} B(y)\e^{-\int_x^y B(z)\f{\d z}{z}}g\Big(t-\log\big(\frac{y}{x}\big),\frac{y}{2}\Big) \f{\d y}{y}\\
&\qquad\qquad+2\int_x^{x\e^t} B(y)\e^{-\int_x^y B(z)\f{\d z}{z}}\partial_tg\Big(t-\log\big(\frac{y}{x}\big),\frac{y}{2}\Big) \f{\d y}{y}\\
&=\Big[\A f(x\e^t)+B(x)f(x\e^t)-2B(x)f\Big(\frac x2\Big)\Big]\e^{-\int_0^t B(x\e^s)\,\d s}\\
&\qquad+2\int_0^{t} \big(B(x\e^\tau)-\e^\tau B(x)\big)\e^{-\int_0^\tau B(x\e^s)\d s}\partial_tg\Big(t-\tau,\frac{x\e^\tau}{2}\Big) \d\tau\\
&\quad\qquad+B(x)\int_0^{t} \e^{-\int_0^\tau B(x\e^s)\d s}\partial_xg\Big(t-\tau,\frac{x\e^\tau}{2}\Big) \d\tau.
\end{align*}
On the one hand, using the second expression of $x\partial_x(\Gamma g)(t,x)$ above we deduce that for $g_1,g_2\in \Cr^1(\Omega_{t_0,K})$ such that $g_1(0,\cdot)=g_2(0,\cdot)=f$ we have
\begin{align*}
&\left\|\Gamma g_1-\Gamma g_2\right\|_{\Cr^1}\\
&\leq t_0 \sup_{(0,K)}\!B\,\Big(2\left\|g_1-g_2\right\|_\infty+2(2+\e^{t_0})\left\|\partial_tg_1-\partial_tg_2\right\|_\infty+\left\|\partial_xg_1-\partial_xg_2\right\|_\infty\Big)\\
&\leq 2t_0(2+\e^{t_0})\sup_{(0,K)}\!B\,\left\|g_1-g_2\right\|_{\Cr^1}.
\end{align*}
Thus $\Gamma$ is a contraction for $t_0$ small enough and this guarantees that the fixed point $\varphi$ necessarily belongs to $\Cr^1([0,\infty)\times(0,\infty)).$
On the other hand, using the first expression of $x\partial_x(\Gamma g)(t,x)$ we have
\[\partial_t(\Gamma g)(t,x)-x\partial_x(\Gamma g)(t,x)=B(x)\bigg(2g\Big(t,\frac x2\Big)-\Gamma g(t,x)\bigg)\]
and accordingly the fixed point satisfies $\partial_t\varphi=\A\varphi$.
\end{proof}

From now on, we assume that the division rate $B$ satisfies~\eqref{hyp:Bsol}.
From Proposition~\ref{lm:fixedpoint} we deduce that Equation~\eqref{eq:dual} generates a positive semigroup on $\dot\B$ by setting for any $t\geq0$ and $f\in\B_{loc}(0,\infty)$
\[M_tf:=\varphi(t,\cdot).\]

\begin{coro}\label{cor:Mtdroite}
The family $(M_t)_{t\geq0}$ defines a semigroup of positive operators on~$\Bloc(0,\infty).$
If $f\in D(\A)$ then the function $(t,x)\mapsto M_tf(x)$ is continuously differentiable on $(0,\infty) \times (0,\infty)$ and satisfies \[\p_tM_tf(x) = \A M_tf(x) = M_t\A f(x).\]
Moreover the subspaces $\dot\B$ and $\Cdot$ are invariant under $M_t$, and for any $f\in\dot\B$ and any $t\geq0$
\[\left\|M_tf\right\|_{\dot\B}\leq\e^t\left\|f\right\|_{\dot\B}.\]
\end{coro}

\medskip

\begin{proof}

The semigroup property $M_{t+s}=M_tM_s$ follows from the uniqueness of the fixed point in the proof of Proposition~\ref{lm:fixedpoint},
$(t,x)\mapsto M_{t+s}f(x)$ and $(t,x)\mapsto M_t(M_sf)(x)$ being both solutions with initial distribution $M_sf\in\Bloc(0,\infty)$.

\smallskip

The positivity of $M_t$ is given by Proposition~\ref{lm:fixedpoint}.

\smallskip

Proposition~\ref{lm:fixedpoint} also provides the regularity of $(t,x)\mapsto M_tf(x)$ when $f\in\Bloc\cap \Cr^1(0,\infty)$, as well as the identity $\p_tM_tf=\A M_tf$.
Besides, if $f\in D(\A)$ then $\A f\in\Bloc$ and~\eqref{eq:dtGammag} with $g(t,x)=M_tf(x)$ ensures, still by uniqueness of the fixed point, that $\partial_tM_tf=M_t\mathcal Af.$

\smallskip

Simple calculations provide that if $f(x)=x$ then the fixed point of $\Gamma$ is given by $M_tf(x)=x\e^t$.
Together with the positivity of $M_t$ it guarantees that $\|M_tf\|_{\dot\B}\leq\e^t\|f\|_{\dot\B}$ for any $f$ in~$\dot\B$ since $-\|f\|_{\dot\B}\,x\leq f(x)\leq\|f\|_{\dot\B}\,x$ for all $x>0$.
In particular $\dot\B$ is invariant under $M_t$, and $\Cdot$ also by virtue of Proposition~\ref{lm:fixedpoint}.

\end{proof}

We give now another useful property of the positive operators $M_t$, namely that they preserve increasing pointwise limits.

\begin{lem} \label{lem:limitecroissante}
Let $f\in\Bloc(0,\infty)$ and let $(f_n)_{n\in \N}\subset\Bloc(0,\infty)$ be an increasing sequence that converges pointwise to $f$, {\it i.e} for all $x>0$
\[f(x)=\lim_{n\to\infty}\!\uparrow f_n(x).\]
Then for all $t\geq0$ and all $x>0$
\[M_tf(x)=\lim_{n\to\infty} M_t f_n(x).\]
\end{lem}

\begin{proof}
Let $f$ and $(f_n)_{n\in \N}$ satisfy the assumptions of the lemma.
For all $t\geq0$, the positivity of $M_t$ ensures that the sequence $(M_tf_n)_{n\in\N}$ is increasing and bounded by $M_tf$.
Denote by $g(t,x)$ the limit of $M_tf_n(x)$.
Using the monotone convergence theorem, we get by passing to the limit in
\[M_tf_n(x) = f_n(x\e^t)\e^{-\int_0^t B(x\e^s)\d s} + 2\int_0^t B(x\e^\tau)\e^{-\int_0^\tau B(x\e^s)\d s}M_{t-\tau}f_n\Big(\frac{x\e^\tau}{2}\Big)\d \tau\]
that
\[g(t,x) = f(x\e^t)\e^{-\int_0^t B(x\e^s)\d s} + 2\int_0^t B(x\e^\tau)\e^{-\int_0^\tau B(x\e^s)\d s}g\Big(t-\tau,\frac{x\e^\tau}{2}\Big)\d \tau.\]
By uniqueness property we deduce that $g(t,x)=M_tf(x)$.
\end{proof}

\subsection{Construction of a measure solution}

Using the results in Section~\ref{subsec:adj}, we define a left action of the semigroup $(M_t)_{t\geq0}$ on $\dot\M.$
To do so we first set for $t\geq0$, $\mu\in\dot\M_+$, and $A\subset(0,\infty)$ Borel set
\[(\mu M_t)(A):=\int_{(0,\infty)}\! M_t\1_A\,\d\mu\]
and verify that $\mu M_t$ such defined is a positive measure on $(0,\infty)$.

\begin{lem}
For all $\mu \in\Mdot_+$ and all $t\geq 0,\ \mu M_t$ defines a positive measure.
Additionally $\mu M_t\in\dot\M_+$ and for any $f\in\dot\B$
\[\langle\mu M_t,f\rangle=\langle\mu,M_tf\rangle.\]
\end{lem}

\begin{proof}

Let $\mu\in\dot\M_+$ and $t\geq0.$
We first check that $\mu M_t$ is a positive measure.

\smallskip

Clearly $\mu M_t(A)\geq0$ for any Borel set $A$, and $\mu M_t(\varnothing) =  \intRpe M_t\mathbf{0}\,\d \mu = 0$.

\smallskip

Let $(A_n)_{n\in \N}$ be a countable sequence of disjoint Borel sets of $\Rpe$ and define $f_n = \sum_{k=0}^n \1_{A_k} =  \mathds{1}_{\bigsqcup_{k=0}^n A_k}$.
For every integer $n,$ one has
\[\mu M_t \left(\bigsqcup_{k=0}^n A_k\right)=\int_{(0,\infty)}\! M_tf_n\,\d \mu = \sum_{k=0}^n \int_{(0,\infty)}\! M_t(\1_{A_k})\d \mu = \sum_{k=0}^n \mu M_t(A_k).\]
The sequence $(f_n)_{n\in\N}$ is increasing and its pointwise limit is $f=\1_{\bigsqcup_{k=0}^\infty A_k}$, which belongs to $\Bloc(0,\infty)$.
We deduce from Lemma~\ref{lem:limitecroissante} and the monotone convergence theorem that
\[\lim_{n\to\infty}\mu M_t \left(\bigsqcup_{k=0}^n A_k\right)=\lim_{n\to\infty}\int_{(0,\infty)}\! M_tf_n\,\d \mu = \int_{(0,\infty)}\! M_tf\,\d \mu = \mu M_t\left(\bigsqcup_{k=0}^\infty A_k\right)\]
where the limit lies in $[0,+\infty].$
This ensures that
\[\mu M_t\left(\bigsqcup_{k=0}^\infty A_k\right)=\sum_{k=0}^\infty \mu M_t(A_k)\]
and $\mu M_t$ thus satisfies the definition of a positive measure.

\medskip

By definition of $\mu M_t$, the identity $\langle\mu M_t,f\rangle=\langle\mu,M_tf\rangle$ is clearly true for any simple function $f$.
Since any nonnegative measurable function is the increasing pointwise limit of simple functions, Lemma~\ref{lem:limitecroissante} ensures that it is also valid in $[0,+\infty]$ for any nonnegative $f\in\Bloc(0,\infty)$.
Considering $f(x)=x$ we get $\langle\mu M_t,f\rangle=\langle\mu,f\rangle\e^t<+\infty$, so that $\mu M_t\in\dot\M_+$.
Finally, decomposing $f\in\dot\B$ as $f=f_+-f_-$ we readily obtain that $\langle\mu M_t,f\rangle=\langle\mu,M_tf\rangle$.

\end{proof}

Now for $\mu\in\dot\M$ and $t\geq0$, we naturally define $\mu M_t\in\dot\M$ by
\begin{equation}\label{eq:muMt+-}
\mu M_t=\mu_+M_t-\mu_-M_t.
\end{equation}
It is then clear that the identity $\langle\mu M_t,f\rangle=\langle\mu,M_tf\rangle$ is still valid for $\mu\in\dot\M$ and $f\in\dot\B$.
Notice that this decomposition~\eqref{eq:muMt+-} as the difference of two positive measures is not in general the Jordan decomposition of $\mu M_t$.

\begin{prop}
The left action of $(M_t)_{t\geq0}$ defines a positive semigroup in $\dot\M$, which satisfies for all $t\geq0$ and all $\mu\in\dot\M$
\[\|\mu M_t\|_\Mdot \leq \e^t\|\mu \|_\Mdot.\]
\end{prop}

\begin{proof}
Using the duality relation $\langle\mu M_t,f\rangle=\langle\mu,M_tf\rangle$, it is a direct consequence of Corollary~\ref{cor:Mtdroite}.
\end{proof}

\medskip

Finally we prove that the (left) semigroup $(M_t)_{t\geq0}$ yields the unique measure solutions to Equation~\eqref{eq:mitosis}.

\begin{theorem}
For any $\mu\in\dot\M$, the family $(\mu M_t)_{t\geq0}$ is the unique solution to Equation~\eqref{eq:mitosis}, in the sense of Definition~\ref{def:measure sol}, with initial distribution $\mu.$
\end{theorem}

\begin{proof}

Let $\mu\in\dot\M$.
We first check that $t\mapsto \langle\mu M_t,f\rangle$ is continuous for any $f\in \Cdot$ by writing
\begin{align*}
|\langle\mu M_t,f\rangle-\langle\mu,f\rangle|&\leq \left|\int_0^\infty f(x\e^t)\e^{-\int_0^t B(x\e^s)\d s}-f(x)\,\mu(\d x)\right| \\
&\qquad+ \left|2\int_0^\infty \int_0^t B(x\e^\tau)\e^{-\int_0^\tau B(x\e^s)\d s}M_{t-\tau}f\Big(\frac{x\e^\tau}{2}\Big)\d \tau\,\mu(\d x)\right|\\
&\leq \int_0^\infty \big|f(x\e^t)\e^{-\int_0^t B(x\e^s)\d s}-f(x)\big|\,|\mu|(\d x) \\
&\qquad\qquad+ \e^t\|f\|_{\dot\B}\int_0^\infty (1-\e^{-\int_0^t B(xe^s)\d s})x\,|\mu|(\d x).
\end{align*}
The two terms in the right hand side vanish as $t$ tends to $0$ by dominated convergence theorem and the continuity of $t\mapsto \langle\mu M_t,f\rangle$ follows from the semigroup property.

\medskip

Now consider $f\in D(\A)$.
Integrating $\partial_tM_tf=M_t\A f$ between $0$ and $t$ we obtain for all $x>0$
\[M_tf(x)=f(x)+\int_0^tM_s(\A f)(x)\,\d s.\]
By definition of $D(\A)$, the function $\A f$ belongs to $\dot \B$, so we deduce the inequality $|M_s(\A f)(x)| \leq \|\A f\|\e^sx$ and we can use Fubini's theorem to get by integration against $\mu$
\[\langle\mu,M_tf\rangle=\langle\mu,f\rangle+\left\langle\mu,\int_0^t M_s(\A f)\d s\right\rangle = \langle\mu,f\rangle+\int_0^t \langle\mu ,M_s(\A f)\rangle\d s.\]
The duality relation $\langle\mu M_t,f\rangle=\langle\mu, M_tf\rangle$ then guarantees that $(\mu M_t)$ satisfies~\eqref{eq:mitosismeasure}.

\medskip

It remains to check the uniqueness.
Let $(\mu_t)_{t\geq0}$ be a solution to Equation~\eqref{eq:mitosis} with $\mu_0=\mu$.
Recall that it implies in particular that $t\mapsto\langle\mu_t,f\rangle$ is continuous for any $f\in \Cdot,$ and consequently $t\to\mu_t$ is locally bounded for the norm $\|\cdot\|_{\dot\M}$ due to the uniform boundedness principle.
We want to verify that $\mu_t=\mu M_t$ for all $t\geq0$.
Fix $t>0$ and $f\in \Cr^1_c(0,\infty)$, and let us compute the derivative of the mapping
\[s\mapsto\int_0^s\langle\mu_\tau,M_{t-s}f\rangle\,\d\tau\]
defined on $[0,t]$.
For $0<s<s+h<t$ we have
\begin{align*}
\frac{1}{h}&\bigg[\int_0^{s+h}\langle\mu_\tau,M_{t-s-h}f\rangle\,\d\tau-\int_0^s\langle\mu_\tau,M_{t-s}f\rangle\,\d\tau\bigg]=\frac{1}{h}\int_s^{s+h}\langle\mu_\tau,M_{t-s}f\rangle\,\d\tau\\
&+\int_s^{s+h}\left\langle\mu_\tau,\frac{M_{t-s-h}f-M_{t-s}f}{h}\right\rangle\d\tau+\int_0^s\left\langle\mu_\tau,\frac{M_{t-s-h}f-M_{t-s}f}{h}\right\rangle\d\tau.
\end{align*}
The convergence of the first term is a consequence of the continuity of $\tau\mapsto\langle\mu_\tau,M_{t-s}f\rangle$
\[\frac{1}{h}\int_s^{s+h}\langle\mu_\tau, M_{t-s}f\rangle\,\d\tau\xrightarrow[h\to0]{}\langle\mu_s,M_{t-s}f\rangle.\]
For the second term we use that
\[M_{t-s}f-M_{t-s-h}f = M_{t-s-h}\int_0^h\p_\tau M_\tau f\,\d \tau = M_{t-s-h}\int_0^h M_\tau \A f\,\d \tau\]
to get, since $\tau\mapsto\|\mu_\tau\|_{\dot\M}$ is locally bounded,
\[\left|\int_s^{s+h}\left\langle\mu_\tau,\frac{M_{t-s-h}f-M_{t-s}f}{h}\right\rangle\d\tau\right|\leq h\sup_{\tau\in[0,t]}\left\|\mu_\tau\right\|_\Mdot\, \left\|\A f\right\|_{\dot\B}\e^{t-s}\xrightarrow[h\to0]{}0.\]
For the last term we have, by dominated convergence and using the identity $\partial_tM_tf=\A M_tf,$
\[\int_0^s\left\langle\mu_\tau,\frac{M_{t-s-h}f-M_{t-s}f}{h}\right\rangle\d\tau\xrightarrow[h\to0]{}-\int_0^s\langle\mu_\tau,\A M_{t-s} f\rangle\,\d\tau.\]
Finally we get
\[\frac{\d}{\d s}\int_0^s\langle\mu_\tau,M_{t-s}f\rangle\,\d\tau=\langle\mu_s,M_{t-s}f\rangle-\int_0^s\langle\mu_\tau,\A M_{t-s} f\rangle\,\d\tau=\langle\mu_0,M_{t-s}f\rangle.\]
To obtain the last equality, one has to notice that $f\in D(\A)$, so Corollary~\ref{cor:Mtdroite} ensures that $M_{t-s}f \in D(\A)$ can be used in Definition~\ref{def:measure sol} in place of $f$.
Integrating between $s=0$ and $s=t$ we obtain, since $\mu_0=\mu$,
\[\int_0^t\langle\mu_\tau ,f\rangle\,\d\tau=\int_0^t\langle\mu,M_{t-s}f\rangle\,\d s=\int_0^t\langle\mu M_{\tau},f\rangle\,\d\tau\]
then by differentiation with respect to $t$
\[\langle\mu_t,f\rangle=\langle\mu M_t,f\rangle.\]
By density of $\Cr_c^1(0,\infty)$ in $\Cdot_0$, it ensures that $\mu_t=\mu M_t$.
\end{proof}

\section{Long time asymptotics\label{sec:asymp}}

In~\cite{GreinerNagel} Greiner and Nagel deduce the convergence toward a rotation semigroup from a general result of spectral theory of positive semigroups, valid in $L^p$ spaces with $1\leq p<\infty$~\cite[C-IV, Th. 2.14]{Nagel86}.
In order to apply this theoretical result, they need the equation to be set on a compact size interval $[\alpha,\beta]\subset(0,\infty)$.
In~\cite{vanBrunt2018} van Brunt~{\it et al.} take advantage of the Mellin transform to solve Equation~\eqref{eq:mitosis} explicitly and deduce the convergence in $L^1((0,\infty),x^{1+r}\d x)$ where $r>0$.
However, this method requires the division rate to be monomial, namely $B(x)=x^r$, and $u(0,\cdot)$ to be a $\Cr^2$ function with polynomial decay at $0$ and $\infty.$
In~\cite{BDG} the authors combine general relative entropy inequalities and the Hilbert structure of the space $L^2((0,\infty),x/\U(x)\,\d x)$ to prove that the solutions converge to their orthogonal projection onto the closure of the subspace spanned by the boundary eigenfunctions.
General relative entropy is an elegant and powerful method for deriving the convergence to the Perron eigenfunction, see~\cite{PerthameTransport}, and the novelty of~\cite{BDG} was to apply it to a non-ergodic case.
This method was recently extended to the measure solutions of the growth-fragmentation equation with smooth fragmentation kernel~\cite{Debiec2018}, but this cannot be applied to the singular case of the mitosis kernel.
Our approach rather relies on the general relative entropy of the dual equation~\eqref{eq:dual}.
It allows us to both define a projector on the boundary eigenspace despite the absence of Hilbert structure and prove the weak-* convergence to this projection.
We then use Harris's ergodic theorem to strengthen it into a convergence in weighted total variation norm with exponential speed.
Besides, the exponential rate of convergence can be estimated explicitly in terms of the division rate $B$.

\

\begin{lem}[General Relative Entropy]\label{lem:GRE}
Let $H : \R \to \R$ be a differentiable convex function. Then for all $f\in\dot\B\cap \Cr^1(0,\infty)$ we have
\[\fdt\int_0^\infty x\,\U_0(x)H\left(\frac{M_tf(x)}{x\,\e^t}\right) \d x = -D^H[\e^{-t}M_tf]\leq 0\]
with $D^H$ defined on $\dot{\B}$ by
\begin{align*}
D^H[f] = &\int_0^\infty x B(x)\U_0(x) \\
&\quad \left[H'\left(\frac{f(x)}{x}\right)\left(\frac{f(x)}{x} - \frac{f(x/2)}{x/2}\right) + H\left(\frac{f(x/2)}{x/2}\right) - H\left(\frac{f(x)}{x}\right)\right]\d x.
\end{align*}
\end{lem}

\smallskip

\begin{proof}
For $f\in\dot\B\cap \Cr^1(0,\infty)$ the function $(t,x)\mapsto M_tf(x)$ is continuously differentiable and verifies $\partial_tM_tf(x)=\A M_tf(x)$, see Corollary~\ref{cor:Mtdroite}.
Simple computations then yield, using that $\U$ satisfies~\eqref{Perron},
\begin{align*}
\Big(\ddt-&x\ddx\Big)\left(x\,\U_0(x)H\left(\frac{M_tf(x)}{x\,\e^t}\right)\right) \\
&= \e^{-t}x\,\U_0(x)B(x)H'\left(\frac{M_tf(x)}{x\,\e^t}\right)\left(\frac{M_tf(x/2)}{x/2}-\frac{M_tf(x)}{x}\right)\\
& \qquad - H\left(\frac{M_tf(x)}{x\,\e^t}\right)x\left(4B(2x)\U_0(2x) - B(x)\U_0(x) - \U_0(x)\right)
\end{align*}
and the conclusion follows by integration.
\end{proof}

This result reveals the lack of coercivity of the equation in the sense that the dissipation $D^H[f]$ does not vanish only for $f(x)=\phi_0(x)=x$ but for any function $f$ such that $f(2x)=2f(x)$ for all $x>0.$
In particular all the eigenfunctions $\phi_k$ satisfy this relation, so $D^H[\Re(\phi_k)]=D^H[\Im(\phi_k)]=0$.
More precisely we have the following result about the space
\[X:=\big\{f\in \Cdot^\C\,|\ \forall x>0,\ f(2x)=2f(x)\big\}\]
where $\Cdot^\C$ is the space of complex-valued continuous functions on $(0,\infty)$ such that $\|f\|_{\Cdot^\C}=\sup_{x>0}|f(x)|/x<\infty$.

\begin{lem}\label{lem:X}
We have the identity
\[X =\overline{\mathrm{span}}(\phi_k)_{k\in \Z}\]
and more specifically any $f\in X$ is the limit in $(\Cdot^\C,\|\cdot\|_{\Cdot^\C})$ of a Fejér type sum
\[f = \lim_{N\to\infty}\sum_{k=-N}^N \left(1-\frac{|k|}{N}\right)\langle\nu_k,f\rangle\phi_k.\]
\end{lem}

\begin{proof}
The vector subspace $X$ contains all the $\phi_k$ and is closed in $(\Cdot^\C,\|\cdot\|_{\Cdot^\C})$, so it contains $\overline{\mathrm{span}}(\phi_k)_{k\in \Z}$.

\medskip

To obtain the converse inclusion, we consider $f\in X$ and we write it as
\[f(x) = x\,\theta(\log x)\]
with $\theta:\R\to\C$ a continuous $\log 2$-periodic function.
The Fejér theorem ensures that the Fejér sum, namely the Cesàro means of the Fourier series
\[\sigma_N(\theta)(y) := \frac1N\sum_{n=0}^{N-1}\sum_{k=-n}^n\hat\theta(k)\e^{\frac{2ik\pi}{\log 2}y}= \sum_{k=-N}^N \left(1-\frac{|k|}{N}\right)\hat{\theta}(k)\e^{\frac{2ik\pi}{\log 2}y}\]
where
\[\hat{\theta}(k) = \frac{1}{\log 2}\int_0^{\log 2}\theta(y)\e^{-\frac{2ik\pi}{\log 2}y}\d y\]
converges uniformly on $\R$ to $\theta.$
We deduce that the sequence ${(F_N(f))}_{N\geq1}\subset\mathrm{span}(\phi_k)_{k\in\Z}$ defined by
\[F_N(f)(x):=x\,\sigma_N(\theta)(\log x) = \sum_{k=-N}^N \left(1-\frac{|k|}{N}\right)\hat{\theta}(k)\phi_k(x)\]
converges to $f$ in norm $\|\cdot\|_{_{\Cdot^\C}}.$

\medskip

To conclude it remains to verify that $\hat\theta(k)=\langle\nu_k,f\rangle$.
Since $\int_0^\infty x\,\U(x)\d x=1$ by definition and $\lambda_k\neq\lambda_l$ when $k\neq l$,
we have that $\langle\nu_k,\phi_l\rangle=\delta_{kl}$, the Kronecker delta function.
We deduce that for any positive integer $N$
\[\langle\nu_k,F_N(f)\rangle=\left\{\begin{array}{l}
0 \text{ if } N < |k|, \vspace{2mm}\\
\big(1-\frac{|k|}{N}\big)\hat{\theta}(k) \text{ otherwise}.
\end{array}\right.\]
As a consequence for all $N\geq|k|$ we have
\[|\langle\nu_k,f\rangle - \hat{\theta}(k)| \leq \left\|f - F_N(f)\right\|_{\dot\B} + \frac{|k|}{N}\left\|f\right\|_{\dot\B}\]
and this gives the desired identity by letting $N$ tend to infinity.

\end{proof}

\

We have shown in the proof of Lemma~\ref{lem:X} that the Fejér sums $F_N$ can be extended to $\Cdot^\C$ by setting
\[F_N(f)=\sum_{k=-N}^N \left(1-\frac{|k|}{N}\right)\langle\nu_k,f\rangle\phi_k.\]
The limit when $N\to\infty$, provided it exists, is a good candidate for defining a relevant projection on~$X$.
Using Lemma~\ref{lem:GRE} we prove in the following theorem that the sequence $(F_N(f))_{n\geq1}$ converges in $X$ for any $f\in \Cr^1_c(0,\infty)$,
and that the limit extends into a linear operator $\Cdot_0\to X$ which provides the asymptotic behavior of $(M_t)_{t\geq0}$ on $\Cr_0$.

\begin{theorem}\label{thm}
For any $f\in \Cr^1_c(0,\infty)$ and any $t\geq0$ the sequence
\[F_N(\e^{-t}M_tf)=\sum_{k=-N}^N\Big(1-\frac{|k|}{N}\Big)\langle\nu_k,f\rangle\e^{\frac{2i\pi k}{\log2}t}\phi_k\]
converges in $\Cdot$ and the limit $R_tf$ defines a $\log2$-periodic family of bounded linear operators $R_t:\Cdot_0\to X\cap \Cdot.$
Moreover for all $f\in \Cdot_0$
\[\e^{-t}M_tf-R_tf\xrightarrow[t\to\infty]{}0\]
locally uniformly on $(0,\infty).$
\end{theorem}

\smallskip

Notice that $R_0$ is actually a projector from $\Cdot_0\oplus X$ onto $X$.

\smallskip

\begin{proof}
We know from Corollary~\ref{cor:Mtdroite} that $\e^{-t}M_t$ is a contraction for $\|\cdot\|_{\dot\B}$.
Let $f\in \Cr^1_c(0,\infty).$
We have $\mathcal A f\in \Cdot$ and so $\partial_t(\e^{-t}M_tf)=M_t(\mathcal Af-f)$ is bounded in time in $\Cdot.$
Since $x\partial_xM_tf(x)=\partial_tM_tf(x)-B(x)(2M_tf(x/2)-M_tf(x))$ and $B$ is locally bounded
we deduce that $\e^{-t}\partial_xM_tf$ is locally bounded on $(0,\infty)$ uniformly in $t\geq0$.
So the Arzela-Ascoli theorem ensures that there exists a subsequence of $(\e^{-t-n\log2}M_{t+n\log2}f(x))_{n\geq0}$ which converges locally uniformly on $[0,\infty)\times(0,\infty)$ to a limit $h(t,x)$, with $h(t,\cdot)\in \Cdot$ for all $t\geq0$.
We now use Lemma~\ref{lem:GRE} to identify this limit.
The dissipation of entropy for the convex function $H(x) = x^2,$ denoted $D^2,$ reads
\[D^2[f]=\intRpe xB(x)\mathcal{U}(x)\left|\frac{f(x/2)}{x/2} - \frac{f(x)}{x}\right|^2\d x.\]
The general relative entropy inequality in Lemma~\ref{lem:GRE} guarantees that
\[\intRpe D^2[\e^{-t}M_tf]\d t<+\infty\]
and as a consequence, for all $T>0$,
\[\int_0^T D^2[\e^{-t-n\log 2}M_{t+n\log 2}f]\d t = \int_{n\log 2}^{T+n\log 2} D^2[\e^{-t}M_tf]\d t \xrightarrow[n\to\infty]{} 0.\]
From the Cauchy-Schwarz inequality we deduce that
\[\frac{\e^{-t-n\log 2}M_{t+n\log 2}f(x/2)}{x/2}-\frac{\e^{-t-n\log 2}M_{t+n\log 2}f(x)}{x}\to 0\]
in the distributional sense on $(0,\infty)^2$,
and since $\e^{-t-n\log 2}M_{t+n\log 2}f(x)$ converges locally uniformly to $h(t,x)$ we get that for all $t\geq0$ and $x>0$
\[\frac{h(t,x/2)}{x/2}-\frac{h(t,x)}{x} = 0.\]
This means that $h(t,\cdot)\in X$ for all $t\geq0$, and Lemma~\ref{lem:X} then ensures that
\[h(t,\cdot)=\lim_{N\to\infty}\sum_{k=-N}^N\Big(1-\frac{|k|}{N}\Big)\langle\nu_k,h(t,\cdot)\rangle\phi_k.\]
Since by definition of $\U_k$ we have $\nu_kM_t=\e^{\lambda_kt}\nu_k$, the dominated convergence theorem yields
\[\langle\nu_k,h(t,\cdot)\rangle=\lim_{n\to\infty}\e^{-t-n\log2}\langle\nu_kM_{t+n\log2},f\rangle=\e^{\frac{2ik\pi}{\log2}t}\langle\nu_k,f\rangle\]
and so
\[h(t,\cdot)=\lim_{N\to\infty}\sum_{k=-N}^N\Big(1-\frac{|k|}{N}\Big)\langle\nu_k,f\rangle\e^{\frac{2ik\pi}{\log2}t}\phi_k=\lim_{N\to\infty}F_N(\e^{-t}M_tf).\]
This guarantees that ${(F_N(M_tf))}_{N\geq1}$ is convergent in $\Cdot$.
Its limit denoted by $R_tf$ clearly defines a linear operator $R_t:\Cr^1_c(0,\infty)\to X\cap \Cdot$.
Moreover by local uniform convergence of $\e^{-t-n\log2}M_{t+n\log 2}f$ to $R_tf$ we get that
\[\left\|R_tf\right\|_{\dot\B}\leq\limsup_{n\to\infty}\left\|\e^{-t-n\log 2}M_{t+n\log 2}f\right\|_{\dot\B}\leq\left\|f\right\|_{\dot\B}.\]
Thus $R_t$ is bounded and it extends uniquely to a contraction $\Cdot_0\to X\cap \Cdot$.
The local uniform convergence of $\e^{-t-n\log 2}M_{t+n\log 2}f(x)$ to $R_tf(x)$ for $f\in \Cr^1_c(0,\infty)$ also guarantees the local uniform convergence of $\e^{-t}M_tf-R_tf$ to zero when $t\to+\infty.$
Indeed, letting $K$ be a compact set of $(0,\infty)$ and defining for all $t\geq0$ the integer part $n:=\big\lfloor \frac{t}{\log 2}\big\rfloor$, so that $t':=t-k\log 2\in [0,\log 2]$,
one has
\begin{align*}
&\sup_{x\in K}|\e^{-t}M_tf(x)-R_tf(x)| \\
&= \sup_{x\in K}|\e^{-(n\log 2 + t')}M_{n\log 2 + t'}f(x)-R_{t'}f(x)|\\
&\leq \sup_{x\in K}\sup_{s\in[0,\log 2]}|\e^{-(n\log 2 + s)}M_{n\log 2 + s}f(x)-R_sf(x)|.
\end{align*}
This convergence extends to any $f\in \Cdot_0$ by density.
\end{proof}

Due to the Riesz representation $\dot\M\simeq \Cdot_0'$, we can define a $\log2$-periodic contraction semigroup $R_t$ on $\dot\M$ by setting
for all $\mu\in\dot\M$ and all $f\in \Cdot_0$
\[\langle\mu R_t,f\rangle:=\langle\mu,R_tf\rangle.\]
Theorem~\ref{thm} then yields the weak-* convergence result in Theorem~\ref{thm:main} since $\rho_t=\mu_0R_t$.
The following theorem readily implies the uniform exponential convergence in weighted total variation norm.

\begin{theorem}\label{thm:conv_norm}
Let $r_1,r_2\in\R$ such that $r_1<1<r_2$ and define $w(x)=x^{r_1}+x^{r_2}$.
Then $R_t$ is a bounded endomorphism of $\Cr(w)$ for any $t\geq0$,
and there exist explicit constants $C\geq1$ and $a>0$ such that for all $f\in\B(w)$ and all $t\geq0$
\[\left\|\e^{-t}M_tf-R_tf\right\|_{\B(w)}\leq C\e^{-at}\left\|f-R_0f\right\|_{\B(w)}.\]
\end{theorem}

The proof of this result relies on Harris's ergodic theorem that we apply on discrete time and discrete size sub-problems.
Before giving the details, we thus start by recalling Harris's theorem in a form which can be readily deduced from~\cite[Theorem~1.3]{HairerMattingly}.
Let $\X$ be a measurable space and $P$ a Markov operator on this state space, that is a positive operator in $\B(\X)$ verifying $P\mathbf1=\mathbf1$,
and suppose that this operator $P$ satisfies the two following assumptions:
\begin{enumerate}
\item[(A1)] There exist a function $V:\X\to[1,\infty)$ and constants $K\geq0$ and $\gamma\in(0,1)$ such that
\[PV(x)\leq\gamma V(x)+K\]
for all $x\in\X$.
\item[(A2)] There exist a constant $\alpha\in(0,1)$ and a probability measure $\nu$ such that
\[Pf(x)\geq \alpha\langle\nu,f\rangle\]
for all functions $f\in\B_+(\X)$ and all $x\in\mathcal S:=\{x\in\X\,:\,V(x)\leq R\}$ for some $R>2K/(1-\gamma)$, where $K$ and $\gamma$ are the constants from Assumption~(A1).
\end{enumerate}
Note that the only difference between these assumptions and ~\cite[Assumption~1 and Assumption~2]{HairerMattingly} is that we suppose here that $V$ takes values larger than $1$.

\begin{theorem}[Harris's theorem \cite{HairerMattingly}]\label{thm:HM}
If Assumptions~(A1) and~(A2) hold, then $P$ admits a unique invariant measure $\mu_\star$ and there exist constants $C>0$ and $\varrho\in(0,1)$ such that the bound
\[\left\| P^nf-\langle\mu_\star,f\rangle\right\|_{\B(\X,V)}\leq C\varrho^n\left\| f-\langle\mu_\star,f\rangle\right\|_{\B(\X,V)}\]
holds for every measurable function $f\in\B(\X,V)$.
Moreover, the constants $C$ and $\varrho$ can be chosen in an explicit way in terms of the constants $\alpha,\gamma,K$ and $R$ of Assumptions~(A1) and~(A2).
\end{theorem}

We are now in position to prove Theorem~\ref{thm:conv_norm}.

\begin{proof}
The semigroup $(M_t)_{t\geq0}$ is not a family of Markov operators
so we consider the rescaled semigroup $(P_t)_{t\geq0}$ defined on $\B(0,\infty)$ by
\[P_tf(x):=\frac{M_t(\phi_0 f)(x)}{\e^t\phi(x)}.\]
Since $\phi_0(x)=x$ verifies $M_t\phi_0=\e^t\phi_0$, the family $(P_t)_{t\geq0}$ is clearly a semigroup of Markov operators.
However, since the long time behavior of $(P_t)_{t\geq0}$ consists in persistent oscillations, this continuous time semigroup is not ergodic.
The idea is to apply Harris's tehorem to a discrete time semigroup on a discrete state space.
Let us fix $x>0$ until the end of the proof, and define
\[\X_x:=\big\{y\in(0,\infty): \exists m\in\Z, y=2^mx\big\}.\]
The left action of $P_{\log 2}$ defines an operator on the measures on $\X_x$.
Let us give a rigorous proof of this claim.
It is easily seen in the proof of Proposition~\ref{lm:fixedpoint} that if $f$ vanishes on $\X_z$, then $\Gamma$ leaves invariant the set of functions $g$ such that $g(t,y)=0$ for all $t\geq0$ and $y\in \X_{\e^{-t}z}$.
It implies that the fixed point $M_tf$ belongs to this set, and consequently so does $P_tf$.
In other words, if $y\in \X_{\e^{-t}z}$ then $\supp(\delta_yP_t)\subset\X_z$.
Applying this to $z=x\e^t$ ensures that if $\supp\mu\subset\X_x$ then $\supp(\mu P_t)\subset\X_{\e^tx}$.
Since $\X_{2x}=\X_x$ we deduce that $P_{\log2}$ leaves invariant the elements of $\dot\M$ with support included in $\X_x$.

\medskip

Let us denote by $P$ the operator $P_{\log2}$ seen as a Markov operator on the state space $\X_x$.
We will prove that $P^n$ satisfies Assumptions~(A1) and~(A2) above for some positive integer $n$ and the Lyapunov function $V(x)=x^{q_1}+x^{q_2}$ with $q_1<0<q_2$.
To do so we study the continuous time semigroup $(P_t)_{t\geq0}$.
Its infinitesimal generator is given by
\[\widetilde\A f(x)=xf'(x)+B(x)\big(f(x/2)-f(x)\big)\]
and it satisfies the Duhamel formula
\begin{equation}\label{eq:Duhamel_Pt}
P_tf(x)=f(x\e^t)\e^{-\int_0^t B(x\e^s)\d s} + \int_0^t B(x\e^\tau)\e^{-\int_0^\tau B(x\e^s)\d s}P_{t-\tau}f\Big(\frac{x\e^\tau}{2}\Big)\d \tau
\end{equation}
which is the same as in Proposition~\ref{lm:fixedpoint} but without the factor $2$ before the integral.
We easily check that
\[\widetilde\A V(x)=\big[q_1+(2^{-q_1}-1)B(x)\big]x^{q_1}+\big[q_2+(2^{-q_2}-1)B(x)\big]x^{q_2}.\]
Since $B$ is continuous, $B(x)\to0$ at $x=0$ and $B(x)\to+\infty$ as $x\to+\infty$, we see that for any $\omega\in(0,-q_1)$ the continuous function $\widetilde\A V+\omega V$ is bounded from above, or in other words there exists a constant $K>0$ such that
\[\widetilde\A V\leq -\omega (V-K).\]
Since $\partial_tP_tV=P_t\tilde\A V$ and $\widetilde \A K=0$ we deduce from Gr\"onwall's lemma that
\[P_tV\leq \e^{-\omega t}V+K\]
for all $t\geq0$.
In particular, since $V \geq 1$, this inequality also ensures that $P_tV\leq(\e^{-\omega t}+K)V$ and consequently $\B(V)$ is invariant under $P_t$.
In terms of the original semigroup $(M_t)_{t\geq0}$ this yields that
\begin{equation}\label{eq:boundBw}
\left\|\e^{-t}M_tf\right\|_{\B(w)}\leq(\e^{-\omega t}+K) \left\|f\right\|_{\B(w)}
\end{equation}
for all $t\geq0$ and $f\in\B(w)$ with $w(x)=xV(x)=x^{1+q_1}+x^{1+q_2}$.
As a by-product, it guarantees that $R_t$ is a bounded endomorphism of $\Cr(w)$ since
\[\left\|R_tf\right\|_{\B(w)}\leq\limsup_{n\to\infty}\left\|\e^{-t-n\log 2}M_{t+n\log 2}f\right\|_{\B(w)}\leq K\left\|f\right\|_{\B(w)}.\]
We also deduce that for all integer $n\geq1$ and all $y\in\X_x$
\[P^nV(y)=P_{n\log2}V(y)\leq \gamma V(y)+K\]
with $\gamma=\e^{-\omega\log2}\in(0,1)$.
Assumption~(A1) is thus satisfied by $P^n$ for any integer $n\geq1$, with constants which do not depend on $n$.
We will now prove that Assumption~(A2) is verified for some $n\geq1$ on the sub-level set $\mathcal S=\{y\in\X_x: V(y)\leq R\}$ for some $R>2K/(1-\gamma)$.

\medskip

Fix $R>2K/(1-\gamma)$ and let $\xi_1,\xi_2\in\X_x$ be such that $\mathcal S\subset[\xi_1,\xi_2]$ and $\xi_2> b_1$, where $b_1$ is defined in~\eqref{hyp:B}.
Define $\overline{\mathcal S}:=[\xi_1,\xi_2]\cap\X_x\supset\mathcal S$ and let us index this set by $\xi_1=x_0<x_1<\cdots<x_{n_0}=\xi_2$, meaning that $\overline{\mathcal S}=\{x_0,\cdots,x_{n_0}\}$.
We prove by induction on $n$ that for all $n\in\{0,\cdots,n_0\}$, there exists $c_n>0$ such that for all bounded function $f:\X_x\to[0,\infty)$ and all $y\in\overline{\mathcal S}$
\begin{equation}\label{eq:IH}
P^nf(y)\geq c_n f(\min(2^ny,x_{n_0})).
\end{equation}
It is trivially satisfied for $n=0$ with $c_0=1$.
Assume now that~\eqref{eq:IH} is verified for some $n\in\{0,\cdots,n_0-1\}$.
Iterating once the Duhamel formula~\eqref{eq:Duhamel_Pt} and taking $t=\log2$ we get that for all $f:\X_x\to[0,\infty)$ and all $y\in\overline{\mathcal S}$
\[Pf(y)\geq \eta f(2y) + \eta^2\bigg(\int_0^{\log2} B(y\e^\tau)\d\tau\bigg)f(y)\]
with $\eta:=\exp\big(-\int_{\xi_1}^{2\xi_2}B\big)>0$.
Applying this inequality to $P^nf$ instead of $f$ yields that for all bounded function $f:\X_x\to[0,\infty)$ and all $y\in\overline{\mathcal S}$
\[P^{n+1}f(y)\geq \eta P^nf(2y) + \eta^2\bigg(\int_0^{\log2} B(y\e^\tau)\d\tau\bigg)P^nf(y).\]
If $y<x_{n_0}$ then we get by induction hypothesis~\eqref{eq:IH} that
\[P^{n+1}f(y)\geq \eta P_nf(2y)\geq c_n\eta f(\min(2^{n+1}y,x_{n_0})).\]
If $y=x_{n_0}$ we use that $x_{n_0}=\xi_2>b_1$ to get
\[P^{n+1}f(x_{n_0})\geq\eta^2\bigg(\!\int_0^{\log2} B(\xi_2\,\e^\tau)\d\tau\!\bigg)P^nf(x_{n_0})
\geq\eta^2K_1\xi_2^{\gamma_1}\frac{2^{\gamma_1}-1}{\gamma_1}P^nf(x_{n_0}).\]
We can thus take $c_{n+1}=\min(c_n\eta,\eta^2K_1\xi_2^{\gamma_1}(2^{\gamma_1}-1)/\gamma_1)>0$.
Now that~\eqref{eq:IH} is proved for all $n\in\{0,\cdots,n_0\}$ we take $n=n_0$ and obtain that
\[P^{n_0}f(y)\geq c_{n_0}\langle\delta_{x_{n_0}},f\rangle\]
for all $y\in\mathcal S$ and all bounded $f:\X_x\to[0,\infty)$,
which is Assumption~(A2) with $\alpha=c_{n_0}$ and $\nu=\delta_{x_{n_0}}$.

\medskip

We are in position to apply the Harris's ergodic theorem.
We get the existence of an invariant measure $\mu^x$ on $\X_x$, which integrate $V$, and constants $C\geq1$ and $\varrho\in(0,1)$ such that for all $f\in\B(\X_x,V)$  and all $m\in\N$
\[\sup_{y\in \X_x}\frac{|P^{mn_0}f(y)-\mu^x(f)|}{V(y)}\leq C\varrho^m\sup_{y\in \X_x}\frac{|f(y)-\mu^x(f)|}{V(y)}.\]
Since the constants $\alpha,\gamma,K$ and $R$ are independent of $x$ in our calculations above, Theorem~\ref{thm:HM} ensures that the constants $C$ and $\varrho$ are too.
In particular it implies that $P^{mn_0}f(y)$ converges to $\mu^x(f)$ as $m\to\infty$ for all $y\in\X_x$.
But, defining $t_0=n_0\log2$, we know from Theorem~\ref{thm} that $P_{mt_0}f(y)\to R_0(\phi f)(y)/y$ for all $f\in\B(V)$ as $m\to\infty$.
So we obtain, taking $y=x$ in the left hand side, that for all $f\in\B(V)$ and all $m\in\N$
\[\frac{|P_{m t_0}f(x)-R_0(\phi_0 f)(x)/x|}{V(x)}\leq C\varrho^m\sup_{y\in \X_x}\frac{|f(y)-R_0(\phi_0 f)(y)/y|}{V(y)}.\]
Still using the function $w(x)=xV(x)$, this yields in terms of $(M_t)_{t\geq0}$ that for all $f\in\B(w)$ and all $m\in\N$
\[\frac{|\e^{-mt_0}M_{m t_0}f(x)-R_0f(x)|}{w(x)}\leq C\varrho^m\sup_{y\in(0,\infty)}\frac{|f(y)-R_0f(y)|}{w(y)}.\]
Since we chose any $x\in(0,\infty)$ and the constants $C$ and $\varrho$ are independent of $x$, we finally proved that for all $f\in\B(w)$ and all $m\in\N$
\[\left\|\e^{-mt_0}M_{m t_0}f-R_0f\right\|_{\B(w)}\leq C\varrho^m\left\|f-R_0f\right\|_{\B(w)}.\]
As $R_{mt_0}=R_0$ by periodicity, this gives the result of Theorem~\ref{thm:conv_norm} for discrete times $t=mt_0$.
It easily extends to continuous times due to the bound~\eqref{eq:boundBw}.

\end{proof}

We finish by giving consequences of Theorems~\ref{thm} and~\ref{thm:conv_norm} in terms of mean ergodicity.
Since the limit is $\log2$-periodic we expect by taking the mean in time of the semigroup to get alignment on the Perron eigenfunction.
The results are given in the following corollary for the right semigroup, but again they can readily be transposed to the left action on measures by duality.

\begin{coro}\label{coro}
For any $f\in \Cdot_0$ the two mappings
\[t\mapsto\frac{1}{\log 2}\int_t^{t+\log 2} \e^{-s}M_sf\, \d s \qquad \text{ and } \qquad t\mapsto\frac{1}{t}\int_0^t \e^{-s}M_sf \,\d s\]
converge locally uniformly to $\nu_0 (f)\phi_0$ when $t$ tends to infinity.
Moreover if $w(x)=x^{r_1}+x^{r_2}$ with $r_1<1<r_2$, then there exist constants $C\geq1$ and $a>0$ such that for all $f\in\B(w)$
\[\bigg\|\frac{1}{\log 2}\int_t^{t+\log 2} \e^{-s}M_sf\, \d s-\langle\nu_0,f\rangle\phi_0\bigg\|_{\B(w)}\leq C\e^{-at}\left\|f-\langle\nu_0,f\rangle\phi_0\right\|_{\B(w)}\]
and
\[\bigg\|\frac{1}{t}\int_0^t \e^{-s}M_sf \,\d s-\langle\nu_0,f\rangle\phi_0\bigg\|_{\B(w)}\leq \frac{C}{t}\left\|f-\langle\nu_0,f\rangle\phi_0\right\|_{\B(w)}.\]
\end{coro}

\begin{proof}
Let $f\in \Cr^1_c(0,\infty)$.
On the one hand, since $\e^{-t}M_t$ and $R_t$ are contractions in $\Cdot$ and $\e^{-t}M_tf-R_tf$ tends to zero locally uniformly, we have by dominated convergence theorem the local uniform convergence
\[\frac{1}{\log 2}\int_t^{t+\log 2} \e^{-s}M_sf\, \d s-\frac{1}{\log 2}\int_t^{t+\log 2} R_sf\, \d s\xrightarrow[t\to\infty]{}0.\]
On the other hand, due to the convergence
\[\bigg\|R_sf-\sum_{k=-N}^N\Big(1-\frac{|k|}{N}\Big)\langle\nu_k,f\rangle\e^{\frac{2i\pi k}{\log2}s}\phi_k\bigg\|_{\dot\B}\xrightarrow[N\to\infty]{}0\]
we have that for all $t\geq0$
\[\frac{1}{\log 2}\int_t^{t+\log 2} R_sf \d s = \langle\nu_0,f\rangle\phi_0.\]
This proves the local uniform convergence of the first integral of the lemma for $f\in \Cr^1_c(0,\infty)$, which remains valid for $f\in \Cdot_0$ by density.
As a consequence the Cesàro means
\[\frac1N\sum_{n=0}^{N-1}\int_{n\log2}^{(n+1)\log2}\e^{-s}M_sf\,\d s=\frac{1}{N\log2}\int_0^{N\log2}\e^{-s}M_sf\,\d s\]
also converges to $\langle\nu_0,f\rangle\phi_0$ locally uniformly when $N\to\infty$,
and it implies the convergence of the second mapping in the lemma.
The uniform exponential convergence in weighted supremum norm follows from Theorem~\ref{thm:conv_norm}, integrating between $t$ and $t+\log 2$, and the other one is obtained by integrating between $0$ and~$t$.
\end{proof}
The difference between the two speeds in the previous corollary can be interpreted as the difference in the amount of memory kept from the past.

\section{Conclusion}\label{sec:conclusion}

In this work, we investigated how the cyclic asymptotic behavior of the rescaled solutions of Equation~\eqref{eq:mitosis} exhibited in~\cite{BDG} is transposed in the measure setting.
Despite the absence of Hilbert structure, we managed to build a suitable projection on the boundary spectral subspace by taking advantage of the general relative entropy of the dual equation.
It allowed us to obtain the weak-* convergence of the rescaled measure solutions to a periodic behavior.
Then, using Harris's ergodic theorem on time and space discrete sub-problems, we managed to get uniform exponential convergence in weighted total variation norm.
To our knowledge no estimate on the speed of convergence was known before for such problems.
Here we not only prove that the convergence takes place exponentially fast, but we also obtain explicit estimates on the spectral gap in terms of the division rate $B$.

\medskip

In~\cite{GreinerNagel}, more general growth rates than linear are considered, namely those satisfying $g(2x)=2g(x).$
Our method would work in this case, replacing the weight $x$ by the corresponding dual eigenfunction $\phi(x)$ and the space $X$ by the functions such that $f(2x)/\phi(2x)=f(x)/\phi(x)$.
However, considering such general coefficients, while interesting from mathematical point of view, is not motivated by modeling concerns, that is why we decided to focus on the linear case.
In addition, it makes computations lighter, in particular those of the flow which is explicitly given by an exponential when $g(x) = x$.

\medskip

Our method would also apply to more sophisticated models of mitosis.
For instance the equation considered in~\cite{Gabriel2019} exhibits a similar countable family of boundary eigenelements for the singular mitosis kernel.
To the prize of additional technicalities, our approach can be used to study its long time behavior.


\end{document}